\documentclass[10pt,leqno]{amsart}
   \usepackage[centertags]{amsmath}
   \usepackage{amsfonts}
   \usepackage{amsmath}
   \usepackage{amssymb}
   \usepackage{amsthm}
   \usepackage{newlfont}
   \usepackage[latin1]{inputenc}%%%%para que reconozca las tildes.

\usepackage[toc,page]{appendix}

\usepackage{multirow, bigdelim}

\usepackage{graphicx}

\newcommand{\dosfilas}[2]{
  \ldelim[{2}{2mm}& #1 &\rdelim]{2}{2mm} \\
  & #2 & &  & &
}

\allowdisplaybreaks[3]
\theoremstyle{plain}
\newtheorem{theorem}{Theorem}[section]
\newtheorem{lemma}[theorem]{Lemma}
\newtheorem{proposition}[theorem]{Proposition}
\newtheorem{corollary}[theorem]{Corollary}

\theoremstyle{definition}

\theoremstyle{remark}
\newtheorem{remark}[theorem]{Remark}
\newtheorem*{remark*}{Remark}

\numberwithin{equation}{section}

\newcommand\D{{\mathcal D}}

\newcommand\C{{\mathcal C}}

\newcommand\F{{\mathcal F}}
\newcommand\I{{\mathcal I}}
\newcommand\U{{\mathcal{U}}}

\newcommand\ZZ{{\mathbb Z}}
\newcommand\NN{{\mathbb N}}

\newcommand\Sh{\mbox{\Large $\mathfrak {s}$}}

\title[Symmetric Krall-Hahn polynomials]
{On difference operators for \\ symmetric Krall-Hahn polynomials}

\author{Antonio J. Dur\'an and Manuel D. de la Iglesia}
\address{A. J. Dur\'an\\
Departamento de An\'{a}lisis Matem\'{a}tico \\
Universidad de Sevilla \\
Apdo (P. O. BOX) 1160\\
41080 Sevilla. Spain.}
\email{duran@us.es}
\address{Manuel D. de la Iglesia\\
Instituto de Matem\'aticas \\
Universidad Nacional Aut\'onoma de México \\
Circuito Exterior, C.U.\\
04510, Ciudad de México. Mexico.}
\email{mdi29@im.unam.mx}
\thanks{Partially supported by MTM2015-65888-C4-1-P (Ministerio de Economía y Competitividad),
FQM-262 (Junta de Andalucía), Feder Funds (European
Union) and PAPIIT-DGAPA-UNAM grant IA102617 (México).}



   \date{}
   
\begin{document}
   \maketitle

\begin{abstract}
The problem of finding measures whose orthogonal polynomials are also eigenfunctions of higher-order difference operators have been recently solved by multiplying the classical discrete measures by suitable polynomials. This problem was raised by Richard Askey in 1991. The case of the Hahn family is specially rich due to the number of parameters involved. These suitable polynomials are in the Hahn case generated from a quartet $F_1,F_2,F_3$ and $F_4$ of finite sets of positive integers. In this paper, we show that the order of the corresponding difference operators can be reduced by assuming certain symmetries on both, the parameters of the Krall-Hahn family and the associated quartet $F_1,F_2,F_3$ and $F_4$.
\end{abstract}

\section{Introduction}
H. L. Krall obtained in 1939 a complete classification of the orthogonal polynomials that are also eigenfunctions of a differential operator of order 4 (\cite{Kr2}). Following his pioneering work, orthogonal polynomials that are also eigenfunctions of higher-order differential operators are usually called Krall polynomials. This terminology can be extended
to finite order difference and $q$-difference operators. Krall polynomials are also called bispectral following
the terminology introduced by Duistermaat and Grünbaum (\cite{DG}; see also \cite{GrH1,GrH3}).
Krall polynomials received increasing interest after the 1980's, when it was proved that if one takes the Laguerre $x^\alpha e^{-x}$ or Jacobi weights $(1-x)^\alpha (1+x)^\beta$, assumes that one or two of the parameters $\alpha$ or $\beta$ are nonnegative integers and adds a linear combination of Dirac deltas and their derivatives at the
endpoints of the orthogonality interval, the associated orthogonal polynomials are eigenfunctions of higher-order differential operators
(\cite{koekoe,koe,koekoe2,L1,L2,GrH1,GrHH,GrY,Plamen1,Plamen2,Zh}). In 1991, at the first years of this renewal interest in bispectral orthogonal polynomials, Richard Askey raised
the problem of finding orthogonal polynomials that are also eigenfunctions of higher-order difference operators (see p. 418 of \cite{BGR})).
Askey's question was open until 2012, when one of us found the first examples of measures whose orthogonal polynomials are also eigenfunctions of higher-order difference operators (see \cite{du0}). The recipe given in \cite{du0} to construct Krall discrete polynomials is the following: multiply the classical discrete measures by suitable polynomials. All the conjectures in \cite{du0} have been proved by the authors in a series of papers (\cite{du1,ddI,ddI2}).
The main tool we have used to prove our results is the $\D$-operator method. $\D$-operator is an abstract concept introduced by one of us in \cite{du1} that has proved to be very useful for generating Krall, Krall discrete, and $q$-Krall families of polynomials (\cite{du1,AD,ddI,ddI1,ddI2,dudh,ddI3}).

One of the most important objects in the $\D$-operator method is certain quasi-Casoratian determinant $\Omega$ associated to the corresponding Krall family. This determinant has the form
\begin{equation}\label{casd1}
\Omega (x)=\det \left(\xi_{x-j,m-j}^lY_l(\theta_{x-j})\right)_{l,j=1}^m,
\end{equation}
where $\theta_n$ is the sequence of eigenvalues of the classical discrete family corresponding to its associated second-order difference operator, $\xi^l_{x,y}$ are certain rational functions of $x$ and $Y_l$ are certain polynomials ($\xi^l_{x,y}$ and $Y_l$ also depending on each classical discrete family).

We display here the Hahn case (see \cite{ddI2}). We consider higher-order difference operators of the form
\begin{equation}\label{doho}
D=\sum_{l=s}^rh_l\Sh _l, \quad s\le r,\quad s,r\in \ZZ,
\end{equation}
where $h_l$ are polynomials and $\Sh_l$ denotes the shift operator $\Sh_l(p)=p(x+l)$. If $h_r,h_s\not =0$, then the order of $D$ is $r-s$.

\begin{theorem}[Corollary 6.3 of \cite{ddI2}]\label{cori1}
Let $\rho_{a,b,N} $ be the Hahn weight (see Section \ref{sectKH} for details). Given a quartet of finite sets $\F=(F_1,F_2,F_3,F_4)$ of positive integers (the empty set is allowed), consider the weight $\rho _{a,b,N}^{\F}$ defined by
\begin{equation}\label{udspmi}
\rho _{a,b,N}^{\F}=\prod_{f\in F_1}(b+N+1+f-x)\prod_{f\in F_2}(x+a+1+f)\prod_{f\in F_3}(N-f-x)\prod_{f\in F_4}(x-f)\rho _{a,b,N}.
\end{equation}
Then the orthogonal polynomials with respect to the measure $\rho _{a,b,N}^{\F}$ (whenever they exist) are eigenfunctions of a higher-order difference operator of the form (\ref{doho}) with $r=-s$ and order
\begin{equation}\label{orderi}
2\left(\sum_{f\in F_1,F_2,F_3,F_4}f-\sum_{i=1}^4\binom{k_i}{2}+1\right),
\end{equation}
where $k_i$ denotes the cardinal of $F_i$.
\end{theorem}
In the Krall-Hahn case, the functions which appear in the quasi-Casoratian determinant $\Omega$ (\ref{casd1}) are: $\theta_x=x(x+\tilde a+\tilde b+1)$, the polynomials $Y_l$ are dual Hahn polynomials and we have four different types of rational functions
\begin{align}\label{xit}
\xi^1_{x,j}&=\displaystyle(-1)^j\frac{(x-j-\tilde N)_j}{(x-j+\tilde a+\tilde b+\tilde N+2)_j},\quad
&\xi^3_{x,j}&=1,\\\nonumber
\xi^2_{x,j}&=
\displaystyle\frac{(x-j+\tilde b+1)_j(x-j-\tilde N)_j}{(x-j+\tilde a+1)_j(x-j+\tilde a+\tilde b+\tilde N+2)_j},\quad
&\xi^4_{x,j}&=\displaystyle (-1)^j\frac{(x-j+\tilde b+1)_j}{(x-j+\tilde a+1)_j}.
\end{align}
The new parameters $\tilde a, \tilde b$ and $\tilde N$ are related to $a,b$ and $N$ by the identities $\tilde a=a+\max F_2+\max F_4 +2$, $\tilde b=b+\max F_1+\max F_3 +2$ and $\tilde N=N-\max F_3-\max F_4 -2$ (see \cite{ddI2}).

The order (\ref{orderi}) of the difference operator for which the Krall-Hahn polynomials are eigenfunctions is related to the quasi-Casoratian determinant $\Omega$ (\ref{casd1}) in the following way. $\Omega $ is a rational function and, roughly speaking, we can remove its zeros and poles by using a rational function $S$, so that $S(x)\Omega (x)= -(2x-m+a+b)P(x)$, where $P$ is a polynomial of degree
$2\left( \sum_{f\in F_1,F_2,F_3,F_4}f-\sum_{i=1}^4\binom{k_i}{2}\right) $. This polynomial $P$ gives the order of the difference operator.

\medskip

Once the problem of constructing Krall discrete polynomials has been solved, the mathematical interest turns out to the minimal order of the associated operators. For the case of differential operators, we know since decades that certain symmetries in the weights may reduce the order of the differential operators. For instance, R. and J. Koekoek found in 2000 (see \cite{koekoe2}) a differential operator for the orthogonal polynomials with respect to the weight
\begin{equation}\label{jacobid}
(1-x)^\alpha (1+x)^\beta +M\delta_{-1}+N\delta_1,\quad \alpha, \beta >-1,
\end{equation}
for which they are eigenfunctions; in the case when $\alpha,\beta\in \NN$ and $M>0,N>0$ this order is
$2\alpha+2\beta +6$.
There are computational evidences showing that, except for special values of the
parameters $\alpha$ and $\beta$ and the numbers $M$ and $N$, $2\alpha+2\beta +6$ seems to be the minimum order of a differential operator having the
orthogonal polynomials with respect to (\ref{jacobid}) as eigenfunctions. However this is not true in
general. For instance, when $\alpha=\beta$ (Gegenbauer polynomials) and $M=N$,
we have that $2\alpha+2\beta +6=4\alpha+6.$ But R. Koekoek found a differential operator of
order $2\alpha +4$ for which these Gegenbauer type orthogonal polynomials are eigenfunctions (see \cite{koe}). This reduction of the order of the differential operators under symmetries of the weight also appears in the wider context of Jacobi-Sobolev orthogonal polynomials (see \cite{ddI3}).

The purpose of this paper is to show a similar result for Krall-Hahn polynomials. In this case, we consider the symmetries $a=b$, $F_1=F_2$ and $F_3=F_4$. According to Theorem \ref{cori1} above, there exists a difference operator of order
$
2\left(2\sum_{f\in F_1,F_3}f-2\binom{k_1}{2}-2\binom{k_3}{2}+1\right),
$
having as eigenfunctions the orthogonal polynomials with respect to $\rho _{a,a,N}^{F_1,F_1,F_3,F_3}$. In this paper, we prove that this order can be reduced to
\begin{equation}\label{orderii}
2\left(\sum_{f\in F_1,F_3}f-\binom{k_1}{2}-\binom{k_3}{2}+1\right).
\end{equation}
The contents of this paper are as follows. In Section 2, we prove that the symmetry assumptions imply a surprising factorization of the quasi-Casoratian determinant $\Omega$ (\ref{casd1}). In Section 3, using again the the $\D$-operator approach, we show that this factorization allows a better
choice for the rational function $S$, in the sense that one can construct from
the new $S$ a difference operator of order (\ref{orderii}) having the Krall-Hahn polynomials as eigenfunctions.
Section 4 is devoted to display some examples.

\section{Preliminaries}
As mentioned in the Introduction, our results involve certain nice factorization of the quasi-Casoratian determinant $\Omega$ (\ref{casd1}). This determinant is one of the main characters of the $\D$-operator method. So in this Section, we introduce this quasi-Casoratian determinant and prove its factorization.

We will use the following notation. Given a finite set of positive integers $F=\{f_1,\ldots, f_m\}$ (where $f_i\not=f_j$, $i\not =j$), the expression
\begin{equation*}\label{defdosf}
  \begin{array}{@{}c@{}lccc@{}c@{}}
    & &&\hspace{-1.3cm}{}_{j=1,\ldots , m} \\
    \dosfilas{ z_{f,j} }{f\in F}
  \end{array}
\end{equation*}
inside a matrix or a determinant denotes the submatrix defined by
$$
\begin{pmatrix}
z_{f_1,1}&z_{f_1,2}&\cdots&z_{f_1,m}\\
\vdots&\vdots&\ddots&\vdots\\
z_{f_m,1}&z_{f_m,2}&\cdots&z_{f_m,m}
\end{pmatrix}.
$$

Let $U$ and $V$ be two finite sets of (distinct) nonnegative integers. Denote by $m_1$ the cardinal of $U$, $m_3$ the cardinal of $V$ and $m=2(m_1+m_3)$. Consider $Y_u, u\in U$ and $Z_u, u\in V$ arbitrary polynomials. Let $\xi_{x,j}^h, h=1,3$ be  auxiliary functions and $\theta_x$ a polynomial in $x$. The quasi-Casoratian determinant is then defined by
\begin{align}\label{defom3}
\Omega(x)&=  \left|
  \begin{array}{@{}c@{}lccc@{}c@{}}
    & &&\hspace{-1.9cm}{}_{j=1,\ldots , m} \\
    \dosfilas{ \xi_{x-j,m-j}^1Y_{u}(\theta_{x-j}) }{u\in U} \\
    \dosfilas{(-1)^j \xi_{x-j,m-j}^1Y_{u}(\theta_{x-j})}{u\in U}
    \\
    \dosfilas{ \xi_{x-j,m-j}^3Z_{u}(\theta_{x-j})}{u\in V}\\
    \dosfilas{(-1)^j \xi_{x-j,m-j}^3Z_{u}(\theta_{x-j})}{u\in V}
  \end{array}\hspace{-.6cm}\right| .
\end{align}
For the Krall-Hahn case, $Y_u, u\in U$ and $Z_u, u\in V$ are dual Hahn polynomials, $\xi_{x,j}^h, h=1,3$ are certain rational functions (see (\ref{xis})) and $\theta_x=x(x+a+b+1)$ (the eigenvalues of the second-order difference operator for the Hahn polynomials). However, the factorization of $\Omega$ does not depend on these particular choices.

The first result we prove is that we can factorize $\Omega(x)$ into two determinants of smaller size. Consider the following two determinants associated with $\Omega(x)$:
\begin{align}\label{Om1}
\Omega_1(x)&=  \left|
  \begin{array}{@{}c@{}lccc@{}c@{}}
    & &&\hspace{-2.5cm}{}_{j=1,\ldots , m_1+m_3} \\
    \dosfilas{ \xi_{x-2j+1,m-2j+1}^1Y_{u}(\theta_{x-2j+1}) }{u\in U} \\
    \dosfilas{ \xi_{x-2j+1,m-2j+1}^3Z_{u}(\theta_{x-2j+1})}{u\in V}
  \end{array}\hspace{-.6cm}\right|,
\Omega_2(x)=  \left|
  \begin{array}{@{}c@{}lccc@{}c@{}}
    & &&\hspace{-2.5cm}{}_{j=1,\ldots , m_1+m_3} \\
    \dosfilas{ \xi_{x-2j,m-2j}^1Y_{u}(\theta_{x-2j}) }{u\in U} \\
    \dosfilas{ \xi_{x-2j,m-2j}^3Z_{u}(\theta_{x-2j})}{u\in V}
  \end{array}\hspace{-.6cm}\right|.
\end{align}
Therefore we have the following:
\begin{proposition}\label{Prop1}
Let $\Omega(x)$ be the quasi-Casoratian determinant of size $m\times m$ defined by \eqref{defom3} and $\Omega_1(x),\Omega_2(x)$ the quasi-Casoratian determinants of size $(m_1+m_3)\times (m_1+m_3)$ defined by \eqref{Om1}, for any polynomials $Y_u$, $u\in U$ and  $Z_u$, $u\in V$. Assume that either $\Omega_1(x)\neq0$ or $\Omega_2(x)\neq0$. Then we have the following factorization
\begin{equation}\label{FacOm}
\Omega(x)=(-1)^{\binom{m_1}{2}+\binom{m_3}{2}}2^{m_1+m_3}\Omega_1(x)\Omega_2(x).
\end{equation}
\end{proposition}
\begin{proof}
We make a double set of interchanges of rows and columns. First, we interchange odd columns with even columns. For that we need $\binom{m_1+m_3}{2}$ changes. Second, we interchange the second block of rows of $\Omega(x)$ with the third block of rows. For that we need $m_1(m_1+m_3-1)$ changes. The parity of the sum is the same as the parity of
$
\binom{m_1+m_3}{2}-m_1(m_1+m_3-1)=\binom{m_3}{2}-\binom{m_1}{2},
$
which at the same time has the same parity as $\binom{m_3}{2}+\binom{m_1}{2}$. Therefore, we can write the determinant $\Omega(x)$ as a $2\times2$ block matrix of the form
$$
\Omega(x)=(-1)^{\binom{m_1}{2}+\binom{m_3}{2}}\begin{vmatrix}
A&B\\-A&B
\end{vmatrix},
$$
where $|A|=\Omega_1(x)$ and $|B|=\Omega_2(x)$. Then, using the very well known formula
$$
\begin{vmatrix}
C_{11}&C_{12}\\C_{21}&C_{22}
\end{vmatrix}=|C_{11}||C_{22}-C_{21}C_{11}^{-1}C_{12}|=|C_{22}||C_{11}-C_{12}C_{22}^{-1}C_{21}|,
$$
we have
$
\Omega(x)=(-1)^{\binom{m_1}{2}+\binom{m_3}{2}}|A||2B|=(-1)^{\binom{m_1}{2}+\binom{m_3}{2}}2^{m_1+m_3}\Omega_1(x)\Omega_2(x),
$
since we are assuming that either $|A|\neq0$ or $|B|\neq0$.
\end{proof}

\begin{corollary}\label{coro1}
Assume in the previous Proposition that the coefficients $\xi_{x,j}^h, h=1,3$ satisfy
\begin{equation}\label{condx}
\xi_{x-i,m-i}^h=\xi_{x-i,l}^h\xi_{x-i-l,m-i-l}^h,\quad h=1,3.
\end{equation}
Then
\begin{equation*}\label{FacOmC}
\Omega(x)=\frac{(-1)^{\binom{m_1}{2}+\binom{m_3}{2}}2^{m_1+m_3}}{(\xi_{x-m,1}^1)^{m_1}(\xi_{x-m,1}^3)^{m_3}}\Omega_1(x)\Omega_1(x-1).
\end{equation*}
\end{corollary}
\begin{proof}
For $x=x-1, i=2j-1, l=m-2j$ we have
$\xi_{x-2j,m-2j+1}^h=\xi_{x-2j,m-2j}^h\xi_{x-m,1}^h,$ $h=1,3.$
Applying this to the definition of $\Omega_2(x)$ in \eqref{Om1} we obtain the result.
\end{proof}
In particular, the condition \eqref{condx} is satisfied for the auxiliary numbers $\xi_{x,j}^h$ of the Hahn polynomials (see \eqref{xit}).

\medskip

Another result that we will need is the relation between the minors of the matrix associated with $\Omega(x)$ and the determinant $\Omega(x)$ itself. For any matrix $A$, denote by $(A)_{i,j}$ the minor of $A$, given by the determinant of the matrix made by removing the $i$-th row and the $j$-th column of $A$. We will abuse the notation and denote $\left(\Omega\right)_{i,j}$ the minor of the matrix associated with the determinant $\Omega(x)$ (same for $\Omega_1(x)$ and $\Omega_2(x)$). Then we have

\begin{proposition}\label{Prop2}
Let $\Omega(x), \Omega_1(x)$ and $\Omega_2(x)$ be the determinants defined by \eqref{defom3} and \eqref{Om1}, respectively. Then, for $j=1,\ldots,m_1+m_3$ and $i=1,\ldots,m_1$, we have
\begin{align*}
\frac{\left(\Omega\right)_{i,2j-1}}{(-1)^{j+1}\Omega}&=\frac{\left(\Omega_1\right)_{i,j}}{2\Omega_1},\;\;\frac{\left(\Omega\right)_{i,2j}}{(-1)^{j}\Omega}=\frac{\left(\Omega_2\right)_{i,j}}{2\Omega_2},\;\;\frac{\left(\Omega\right)_{i+m_1,2j-1}}{(-1)^{j+m_1}\Omega}=\frac{\left(\Omega_1\right)_{i,j}}{2\Omega_1},\;\;\frac{\left(\Omega\right)_{i+m_1,2j}}{(-1)^{j+m_1}\Omega}=\frac{\left(\Omega_2\right)_{i,j}}{2\Omega_2},
\end{align*}
while for $j=1,\ldots,m_1+m_3$ and $i=1,\ldots,m_3$, we have
\begin{align}
\nonumber\frac{\left(\Omega\right)_{i+2m_1,2j-1}}{(-1)^{j+m_1+1}\Omega}&=\frac{\left(\Omega_1\right)_{i+m_1,j}}{2\Omega_1},\quad \frac{\left(\Omega\right)_{i+2m_1,2j}}{(-1)^{j+m_1}\Omega}=\frac{\left(\Omega_2\right)_{i+m_1,j}}{2\Omega_2},\\
\label{mom4}\frac{\left(\Omega\right)_{i+2m_1+m_3,2j-1}}{(-1)^{j+m_1+m_3}\Omega}&=\frac{\left(\Omega_1\right)_{i+m_1,j}}{2\Omega_1},\quad \frac{\left(\Omega\right)_{i+2m_1+m_3,2j}}{(-1)^{j+m_1+m_3}\Omega}=\frac{\left(\Omega_2\right)_{i+m_1,j}}{2\Omega_2}.
\end{align}
\end{proposition}
\begin{proof}
We will only prove the first two identities. The rest are very similar. Using the same procedure as in Proposition \ref{Prop1} it is possible to see that
\begin{align*}
\left(\Omega\right)_{i,2j-1}&=(-1)^{\binom{m_1}{2}+\binom{m_3}{2}+j+1}2^{m_1+m_3-1}\left(\Omega_1(x)\right)_{i,j}\Omega_2(x),\\
\left(\Omega\right)_{i,2j}&=(-1)^{\binom{m_1}{2}+\binom{m_3}{2}+j}2^{m_1+m_3-1}\Omega_1(x)\left(\Omega_2(x)\right)_{i,j}.
\end{align*}
It is now enough to use this and \eqref{FacOm}.
\end{proof}

\section{Krall-Hahn polynomials}\label{sectKH}
We start with some basic definitions and facts about Hahn polynomials.
For $a,a+b+1,a+b+N+1\neq-1,-2,\ldots$, we write $(h_n^{a,b,N})_n$ for the sequence of Hahn polynomials defined by
\begin{equation*}\label{HP}
h_n^{a,b,N}(x)=\sum_{j=0}^n\frac{(-x)_j(N-n+1)_{n-j}(a+b+1)_{j+n}}{(2+a+b+N)_n(a+1)_j(n-j)!j!},\quad n\ge 0.
\end{equation*}
Hahn polynomials are eigenfunctions of a second-order difference operator $D_H$
\begin{equation}\label{dopHahn}
D_H(h_n^{a,b,N})=\theta_nh_n^{a,b,N},\quad\theta_n=n(n+a+b+1), \quad n\geq0.
\end{equation}
We also consider the associated function
\begin{equation}\label{sig}
\sigma_{x}=\theta_{x-3/2}-\theta_{x-1/2}=-(2x+a+b-1).
\end{equation}
If we assume that $a,b,a+b,a+b+N+1\neq -1,-2,\ldots,$ and $N+1$ is not a positive integer, then the Hahn polynomials are always orthogonal with respect to a moment functional $\rho_{a,b,N}$.
When $N+1$ is a positive integer, $a,b\neq -1,-2,\ldots,-N$, and $a+b\neq -1,-2,\ldots,-2N-1$, the first $N+1$ Hahn polynomials are orthogonal with respect to the Hahn measure
\begin{equation*}\label{pesoH}
    \rho_{a,b,N}=N!\sum_{x=0}^N\frac{\Gamma(a+x+1)\Gamma(N-x+b+1)}{x!(N-x)!}\delta_x,
\end{equation*}
and they have non-null norms. The discrete measure $\rho_{a,b,N}$ is positive only when $a,b>-1$ or $a,b<-N$.

We associate to a quartet $\F=(F_1,F_2,F_3,F_4)$ of finite sets of positive integers (the empty set is allowed) the Krall-Hahn weight
$\rho _{a,b,N}^{\F}$ (\ref{udspmi}). We proved in \cite{ddI2} that the associated Krall-Hahn orthogonal polynomials (whenever they exist) are eigenfunctions of a higher-order difference operator of the form (\ref{doho}) with $r=-s$ and order given by (\ref{orderi}).
%\begin{equation}\label{orderH}
%2\left(\sum_{f\in F_1,F_2,F_3,F_4}f-\sum_{i=1}^4\binom{k_i}{2}+1\right),
%\end{equation}
%where $k_i$ denotes the cardinal of $F_i,i=1,2,3,4$.

The goal of this paper is to prove that under some symmetries in the choice of the parameters $a,b,$ and the quartet $\F$, we will have that the associated Krall-Hahn orthogonal polynomials are eigenfunctions of a higher-order difference operator of order \emph{less} than \eqref{orderi}. This choice is given by
\begin{equation*}\label{choi}
a=b,\quad F_1=F_2,\quad\mbox{and}\quad F_3=F_4.
\end{equation*}
Therefore, from now on we will remove the dependence on $b$ in the notation and call $F=F_1=F_2$, $G=F_3=F_4$, and $\F=(F,G)$. To simplify the functions $\xi _{x,y}^h $ (\ref{xit}) we set $\hat a=a-\max F-\max G-2$, $\hat N=N+2\max G +2$ so that $\tilde a=a$, $\tilde b=b$ and $\tilde N=N$. Then the Krall-Hahn weight (\ref{udspmi}) is given by
\begin{equation}\label{KHw2}
\rho _{\hat a,\hat N}^{\F}=\prod_{f\in F}(\hat a+\hat N+1+f-x)(x+\hat a+1+f)\prod_{f\in G}(\hat N-f-x)(x-f)\rho _{\hat a,\hat N}.
\end{equation}
We remark that with these choices the four rational functions (\ref{xit}) simplify to
\begin{equation}\label{xis}
\xi_{x,j}^1=(-1)^j\frac{(x-j-N)_j}{(x-j+2a+N+2)_j},\quad \xi_{x,j}^3=1,
\end{equation}
and $\xi_{x,j}^2=(-1)^j\xi_{x,j}^1$, $\xi_{x,j}^4=(-1)^j$.

For convenience, we will write $\xi_{x,j}^1$ in the following way
\begin{equation}\label{xi11}
\xi_{x,j}^1=\frac{N_{x,j}}{D_{x,j}},\quad N_{x,j}=(x-j-N)_j,\quad D_{x,j}=(-1)^j(x-j+2a+N+2)_j.
\end{equation}
Observe that $N_{x,j}$ and $D_{x,j}$ satisfy similar relations as in \eqref{condx}, i.e.
\begin{equation}\label{condxND}
N_{x-i,m-i}=N_{x-i,l}N_{x-i-l,m-i-l},\quad D_{x-i,m-i}=D_{x-i,l}D_{x-i-l,m-i-l}.
\end{equation}
Here we are using the convention
$$
N_{x,i}=\frac{1}{N_{x-i,-i}},\quad D_{x,i}=\frac{1}{N_{x-i,-i}},\quad i\leq-1.
$$
We also need to consider the involution $I$ defined by
\begin{equation*}\label{invI}
I(F)=\{1,2,\ldots, f_k\}\setminus \{f_k-f,f\in F\},
\end{equation*}
where $f_k=\max F$.  We have that $k=f_k-m+1$, where $m$ denotes the cardinal of $I(F)$. We also have the relation
\begin{equation}\label{relF}
\sum_{u\in I(F)}u+kf_k=\sum_{f\in F}f+\binom{f_k+1}{2}.
\end{equation}

To our Krall-Hahn family, we associate the quasi-Casoratian determinants $\Omega ^{a,\F}$, $\Omega_1 ^{a,\F}$ and $\Omega_2 ^{a,\F}$ of the form (\ref{defom3}) and (\ref{Om1}), respectively, where $\xi_{x,j}^h$, $h=1,3$, are given by (\ref{xis}), $U=I(F)$, $V=I(G)$ and $Y_u$, $u\in U$ and $Z_u$, $u\in V$, are given by
\begin{equation*}\label{YZdH}
Y_u(x)=R_u^{-a,-a,2a+N}(x+2a),\quad Z_u(x)=R_u^{-a,-a,-2-N}(x+2a),
\end{equation*}
where $R_u$ is the corresponding dual Hahn polynomial.

Given a rational function $S$ we define the rational functions $M_h$, $h=1,\ldots , m$, by
\begin{equation}\label{defMh}
M_h(x)=\sum_{j=1}^{m}(-1)^{h+j}\xi_{x,m-j}^{i(h)}S(x+j)\left(\Omega(x+j)\right)_{h,j},
\end{equation}
where 
%$i(h)=1$ for $h=1,\ldots, m_1$, $i(h)=2$ for $h=m_1+1,\ldots , 2m_1$, $i(h)=3$ for $h=m_1+m_2+1,\ldots , 2m_1+m_3$ and $i(h)=4$ for $h=2m_1+m_3+1,\ldots , m$.
$$
i(h)=\begin{cases}
1, &h=1,\ldots , m_1,\\
2, &h=m_1+1,\ldots , 2m_1,\\
3, &h=2m_1+1,\ldots , 2m_1+m_3,\\
4, &h=2m_1+m_3+1,\ldots , m.
\end{cases}
$$
We are interested in rational functions $S$ satisfying the following three hypotheses:
\begin{align}\label{ass0}
&\mbox{$S(x)\Omega (x)$ is a polynomial in $x$.}\\
\label{ass1}&\mbox{There exist $\tilde{M}_1,\ldots,\tilde{M}_{m},$  polynomials in $x$ such that}\\
\nonumber&\hspace{1cm}\mbox{$M_h(x)=\sigma_{x+1}\tilde{M}_h(\theta_x)$, $h=1,\ldots,m.$}\\
\label{ass2}&\mbox{There exists a polynomial $P_S$ such that }\\
\nonumber&\hspace{1cm}\mbox{$\displaystyle P_S(\theta_x)-P_S(\theta_{x-1})=S(x)\Omega(x)+S(x+m)\Omega(x+m)$.}
\end{align}
Notice that (\ref{ass1}) implies that $M_h$ are also polynomials.

In Lemma 5.2 of \cite{ddI2} a rational function $S(x)$ is chosen in such a way that the previous hypotheses hold.  This $S(x)$ is valid for any values of the parameters $a,b,N$ and any sets of positive integers $F_1, F_2, F_3, F_4$, and the order of the associated difference operator is \eqref{orderi}. When the symmetries $a=b$ and $F_1=F_2$, $F_3=F_4$ are assumed, a better rational function $S(x)$ can be found, in the sense that the associated difference operator has order \emph{lower} than \eqref{orderi}. More precisely, we prove that this order is given by (\ref{orderii}).

Our rational function $S$ has the form
\begin{equation}\label{SSx}
S(x)=\frac{\sigma_{x-\frac{m-1}{2}}R(x)}{\Omega(x)},
\end{equation}
where $R$ is a suitable polynomial (see \eqref{EQR2} below). Observe that with this choice we can write $M_h(x)$ in \eqref{defMh} as
\begin{equation}\label{Mhx}
M_h(x)=\sum_{j=1}^{m}(-1)^{h+j}\xi_{x,m-j}^{i(h)}\sigma_{x+j-\frac{m-1}{2}}R(x+j)\frac{\left(\Omega(x+j)\right)_{h,j}}{\Omega(x+j)}.
\end{equation}
Therefore, from Proposition \ref{Prop2} it is easy to see that
$$
M_h(x)=M_{h+m_1}(x),\quad h=1,\ldots,m_1,\quad \mbox{and}\quad M_{2m_1+h}(x)=M_{2m_1+m_3+h}(x),\quad h=1,\ldots,m_3.
$$
Hence, the only relevant functions are the first $m_1$ and the first $m_3$ after the first $2m_1$ functions.

If hypotheses \eqref{ass0}, \eqref{ass1} and \eqref{ass2} hold, the $\D$-operator method provides a difference operator $D_S$ for which the Krall-Hahn polynomials are eigenfunctions (see Theorem 3.1 of \cite{ddI2}). This difference operator can be explicitly constructed:
\begin{equation}\label{Dq2}
D_{S}=\frac{1}{2}P_S(D_H)+\sum_{h=1}^{m_1}\tilde M_h(D_H)\left(\D_1+\D_2\right) Y_h(D_H)+\sum_{h=1}^{m_3}\tilde M_{2m_1+h}(D_H)\left(\D_3+\D_4\right) Z_h(D_H),
\end{equation}
where $D_H$ is the second-order difference operator for the Hahn polynomials (\ref{dopHahn}), $P_S$ is defined in \eqref{ass2}  and $\D_1,\D_2,\D_3,\D_4$ are
the corresponding $\D$-operators for the Krall polynomials defined by
\begin{align*}
\D_1&=\frac{a+b+1}{2}I+x\nabla,\quad &\D_2&=\frac{a+b+1}{2}I+(x-N)\Delta,\\
\D_3&=\frac{a+b+1}{2}I+(x+a+1)\Delta,\quad &\D_4&=\frac{a+b+1}{2}I+(x-b-N-1)\nabla.
\end{align*}

The orthogonal polynomials with respect to the Krall-Hahn weight \eqref{KHw2} are eigenfunctions of \eqref{Dq2}.

In order to introduce the polynomial $R$ in \eqref{SSx} we need the following technical Lemma (which we prove in the Appendix).
\begin{lemma}\label{lpqR} Let $p$ and $q$ be the polynomials defined by
\begin{align}
\label{ppp}p(x)&=\prod_{j=1}^{m_1-1}D_{x-1,2j}N_{x-m+2j+1,2j},\\
\label{qqq}q(x)&=\prod_{i=1}^{\lfloor \frac{m_1+m_3}{2}\rfloor}\left(\prod_{j=1}^{m-4i+1}\sigma_{x-m+2i+j}\right),
\end{align}
where $N_{x,j}, D_{x,j}$ are given by \eqref{xi11}. Let $H$ be the rational function defined by
\begin{equation}\label{HH2}
H(x)=\frac{\sigma_{x-m_1-m_3+1}\left(D_{x-1,m-1}\right)^{m_1}}{p(x)q(x)\left(N_{x-m+1,1}\right)^{m_1}}.
\end{equation}
Then $H(x)\Omega_1^{a,\F}(x)$ is a polynomial, where $\Omega_1^{a,\F}(x)$ is defined by \eqref{Om1}.
\end{lemma}

We then define the polynomial $R$ as the (monic) solution of the difference equation
\begin{equation}\label{EQR2}
\sigma_{x-\frac{m-1}{2}}R(x)+\sigma_{x+1-\frac{m-1}{2}}R(x+1)=H(x)\Omega_1^{a,\F}(x).
\end{equation}
We are now ready to establish the main result of this paper.

\begin{theorem}\label{Th2}
With the previous choices of the functions $\Omega ^{a,\F }, S$ and $M_h, h=1,\ldots,m$, the three hypotheses \eqref{ass0}, \eqref{ass1} and \eqref{ass2} hold. Moreover, the order of the associated difference operator $D_S$ (\ref{Dq2}) is given by
\begin{equation*}\label{ordD2}
2\left(\sum_{f\in F_1,F_3}f-\binom{k_1}{2}-\binom{k_3}{2}+1\right).
\end{equation*}
\end{theorem}

To prove the Theorem we need to introduce some more tools. First of all, we need the involution that characterizes the subring $\mathbb{R}[\theta_x]$ in $\mathbb{R}[x]$. This involution is given by
\begin{equation}\label{inv}
\I^{a+b}\big(p(x)\big)=p\big(-(x+a+b+1)\big),\quad p\in\mathbb{R}[x].
\end{equation}
Clearly we have $\I^{a+b}(\theta_x)=\theta_x$. Hence every polynomial in $\theta_x$ is invariant under the action of $\I^{a+b}$. Conversely, if $p\in\mathbb{R}[x]$ is invariant under $\I^{a+b}$, then $p\in\mathbb{R}[\theta_x]$. We also have that if $p\in\mathbb{R}[x]$ is skew invariant, i.e. $\I^{a+b}(p)=-p,$ then $p$ is divisible by $\theta_{x-1/2}-\theta_{x+1/2}$, and the quotient belongs to $\mathbb{R}[\theta_x]$. We remark here that, in the case of Hahn polynomials and from the definition of $\theta_x$ and $\sigma_x$, we have that $\sigma_{x+1}=\theta_{x-1/2}-\theta_{x+1/2}$. Observe that $\sigma_{x+1}$ is skew invariant itself.

According to definitions \eqref{inv}, (\ref{dopHahn}), (\ref{sig}), (\ref{condxND}), (\ref{ppp}) and (\ref{qqq})   we have
\begin{equation}\label{Iprop1}
\I^{2a+i}(\theta_{x-j})=\theta_{x+i+j},\quad \I^{2a+i}(\sigma_{x-j})=-\sigma_{x+i+j+2},
\end{equation}
\begin{equation}\label{IpropND}
\I^{2a+i}(N_{x-j-s,m-s})=D_{x+m+i+j,m-s},\quad \I^{2a+i}(D_{x-j-s,m-s})=N_{x+m+i+j,m-s},
\end{equation}
\begin{equation}\label{invpq}
\I^{2a+i}\left(p(x-j)\right)=p(x+i+j+m),\quad \I^{2a+i}\left(q(x-j)\right)=(-1)^{\binom{m_1+m_3+2}{2}+1}q(x+i+j+m).
\end{equation}
We finally need the following two technical Lemmas.
\begin{lemma}\label{lgp2} Let $\Psi_j$, $j=1,\ldots, m$, rational functions with, at most, simple poles. Assume that
\begin{enumerate}
\item for all polynomial $p$, the function $L_p(x)=\sum_{j=1}^m(-1)^jp(\theta_{x-j})\Psi_j(x)$ is a polynomial;
\item if $M(x)=\sum_{j=1}^m(-1)^j\Psi_j(x+j)$, then $\I^{a+b}(M)=-M$, $h=1,\ldots , m$.
\end{enumerate}
Then $M$ is also a  polynomial. If, in addition, we assume that the degree of $L_{x^g}$ is at most $2g+\deg (L_1)$, then $M$ is a polynomial of degree at most $\deg (L_1)$.
\end{lemma}
This Lemma is an update of Lemma 3.4 of \cite{ddI2}, which it can not be applied directly to the situation of Krall-Hahn polynomials since the addends in the functions $M_h$ (\ref{defMh}) are not polynomials in general (they may have poles at numbers of the form $-(a+b+1+j)/2$, $j\in \NN$). We prove it in the Appendix.

\begin{lemma}\label{lemaP}
Let $Y_u, u\in U,$ and $Z_u, u\in V,$ be nonzero polynomials such that $\deg Y_u=u$ and $\deg Z_u=u$. For real numbers $a, N,$ consider the rational function $P$ defined by
\begin{equation}\label{PPP}
P(x)=  \frac{\left|
  \begin{array}{@{}c@{}lccc@{}c@{}}
    & &&\hspace{-2.5cm}{}_{j=1,\ldots , m_1+m_3} \\
    \dosfilas{N_{x-2j+1,m-2j}D_{x-1,2j-2}Y_{u}(\theta_{x-2j+1}) }{u\in U} \\
    \dosfilas{ Z_{u}(\theta_{x-2j+1})}{u\in V}
  \end{array}\hspace{-.5cm}\right|}{p(x)q(x)},
\end{equation}
where $p$ and $q$ are the polynomials (\ref{ppp}) and (\ref{qqq}), respectively.
If $v-u+a+N+1\neq0$ for $u\in U, v\in V$, then $P$ in \eqref{PPP} is a polynomial of degree
\begin{equation}\label{degP}
d=2\left(\sum_{u\in U, V}u-\binom{m_1}{2}-\binom{m_3}{2}\right).
\end{equation}
\end{lemma}
This Lemma is a particular instance of Lemma 5.1 of \cite{ddI2} (see also \cite{dudh}).

\medskip

We are now ready to proof Theorem \ref{Th2}.

\begin{proof}[Proof of the Theorem \ref{Th2}] First we will prove that $M_h(x), h=1,\ldots,m_1,$ in \eqref{Mhx} can be written as
\begin{equation}\label{MhxN1}
M_h(x)=\frac{1}{2}\sum_{j=1}^{m_1+m_3}(-1)^{h+j}\xi_{x,m-2j+1}^1H(x+2j-1)\left(\Omega_1^{a,\F}(x+2j-1)\right)_{h,j},
\end{equation}
while $M_{2m_1+h}(x), h=1,\ldots,m_3,$ can be written as
\begin{equation}\label{MhxN2}
M_{2m_1+h}(x)=\frac{1}{2}\sum_{j=1}^{m_1+m_3}(-1)^{h+j+m_1}\xi_{x,m-2j+1}^3H(x+2j-1)\left(\Omega_1^{a,\F}(x+2j-1)\right)_{m_1+h,j}.
\end{equation}
We will only prove \eqref{MhxN1}. The case \eqref{MhxN2} is similar using Proposition \ref{Prop2}. First observe from the definition of $M_h(x)$ in \eqref{Mhx} the strong dependence on expressions of the form $\frac{\left(\Omega^{a,\F}(x)\right)_{h,j}}{\Omega^{a,\F}(x)}$. Using Proposition \ref{Prop2} these expressions only depend on $\frac{\left(\Omega_1^{a,\F}(x)\right)_{h,j}}{\Omega_1^{a,\F}(x)}$ and $\frac{\left(\Omega_2^{a,\F}(x)\right)_{h,j}}{\Omega_2^{a,\F}(x)}.$ Now using Corollary \ref{coro1} and the property
$$
\xi_{x,m-2j+1}^1\frac{\left(\Omega_1^{a,\F}(x)\right)_{h,j}}{\Omega_1^{a,\F}(x)}=\xi_{x,m-2j}^1\frac{\left(\Omega_2^{a,\F}(x)\right)_{h,j}}{\Omega_2^{a,\F}(x)},\quad h=1,\ldots,m_1,
$$
we see that can group every two consecutive addends of \eqref{Mhx} into one expression, i.e.
\begin{align*}
M_h(x)=&\frac{1}{2}\sum_{j=1}^{m_1+m_3}(-1)^{h+j}\xi_{x,m-2j+1}^1\left(\Omega_1^{a,\F}(x+2j-1)\right)_{h,j}\times\\
\qquad \qquad \qquad &\times\left(\frac{\sigma_{x+2j-1-\frac{2m_1-1}{2}}R(x+2j-1)+\sigma_{x+2j-\frac{2m_1-1}{2}}R(x+2j)}{\Omega_1^{a,\F}(x+2j-1)}\right).
\end{align*}
From \eqref{EQR2} we get \eqref{MhxN1}. Using \eqref{xi11}, \eqref{condxND}, \eqref{HH2} and Lemma \ref{lemaP}, we can write $M_h$ as
\begin{align}\label{MhxN1aux}
M_h(x)=&\frac{1}{2}\sum_{j=1}^{m_1+m_3}(-1)^{j+h}\frac{N_{x,m-2j}D_{x+2j-2,2j-2}\sigma_{x-m_1-m_3+2j}}{p(x+2j-1)q(x+2j-1)}\times\\
\nonumber&\qquad\times\left|
  \begin{array}{@{}c@{}lccc@{}c@{}}
    & &&\hspace{-2cm}{}_{u\neq h, r\neq j} \\
    \dosfilas{D_{x+2j-2,2r-2}N_{x+2j-2r,m-2r}Y_{u}(\theta_{x+2j-2r}) }{u\in U}\\
    \dosfilas{ Z_{u}(\theta_{x+2j-2r})}{u\in V}
  \end{array}\hspace{-.5cm}\right|.
\end{align}

Now let us prove the three hypotheses \eqref{ass0}, \eqref{ass1} and \eqref{ass2}. The first one is trivial from the definition of $S(x)$ in \eqref{SSx} since $S(x)\Omega^{a,\F}(x)=\sigma_{x-\frac{m-1}{2}}R(x)$ and from Lemma \ref{lpqR} we know that $R(x)$ is indeed a polynomial. For the second \eqref{ass1} it is enough to see that $\I^{2a}(M_h(x))=-M_h(x)$ where $\I$ is the involution defined by \eqref{inv} (see the proof of Lemma 5.2 of \cite{ddI2}). Using \eqref{Iprop1}, \eqref{IpropND}, \eqref{invpq}  and the definition of $M_h(x)$ in \eqref{MhxN1aux}, we have
\begin{align*}
\I^{2a}(M_h(x))&=\frac{1}{2}\sum_{j=1}^{m_1+m_3}\frac{(-1)^{j+h}}{(-1)^{\binom{m_1+m_3+2}{2}}}\frac{D_{x+m-2j,m-2j}N_{x,2j-2}\sigma_{x+m_1+m_3-2j+2}}{p(x-2j+m+1)q(x-2j+m+1)}\times\\
&\qquad\times\left|
  \begin{array}{@{}c@{}lccc@{}c@{}}
    & &&\hspace{-2cm}{}_{u\neq h, r\neq j} \\
    \dosfilas{D_{x+m-2j,m-2r}N_{x-2j+2r,2r-2}Y_{u}(\theta_{x-2j+2r}) }{u\in U}\\
    \dosfilas{ Z_{u}(\theta_{x-2j+2r})}{u\in V}
  \end{array}\hspace{-.5cm}\right|\\
  &=\frac{1}{2}\sum_{j=1}^{m_1+m_3}\frac{(-1)^{j+h+m_1+m_3-1}}{(-1)^{\binom{m_1+m_3+2}{2}}}\frac{D_{x+2j-2,2j-2}N_{x,m-2j}\sigma_{x-m_1-m_3+2j}}{p(x+2j-1)q(x+2j-1)}\times\\
&\qquad\times\left|
  \begin{array}{@{}c@{}lccc@{}c@{}}
    & &&\hspace{-2cm}{}_{u\neq h, r\neq j} \\
    \dosfilas{D_{x+2j-2,m-2r}N_{x-m+2j-2+2r,2r-2}Y_{u}(\theta_{x-m+2j-2+2r}) }{u\in U}\\
    \dosfilas{ Z_{u}(\theta_{x-m+2j-2+2r})}{u\in V}
  \end{array}\hspace{-.5cm}\right|\\
  &=\frac{1}{2}\frac{(-1)^{m_1+m_3-1}(-1)^{\binom{m_1+m_3-1}{2}}}{(-1)^{\binom{m_1+m_3+2}{2}}}\sum_{j=1}^{m_1+m_3}(-1)^{j+h}\frac{D_{x+2j-2,2j-2}N_{x,m-2j}\sigma_{x-m_1-m_3+2j}}{p(x+2j-1)q(x+2j-1)}\times\\
&\qquad\times\left|
  \begin{array}{@{}c@{}lccc@{}c@{}}
    & &&\hspace{-2cm}{}_{u\neq h, r\neq j} \\
    \dosfilas{D_{x+2j-2,2r-2}N_{x+2j-2r,m-2r}Y_{u}(\theta_{x+2j-2r}) }{u\in U}\\
    \dosfilas{ Z_{u}(\theta_{x+2j-2r})}{u\in V}
  \end{array}\hspace{-.5cm}\right|\\
  &=(-1)^{(m_1+m_3)(m_1+m_3+1)+3}M_h(x)=-M_h(x).
\end{align*}
Note that we renamed the index $j$ in the second step and that we interchanged all of the columns in the determinant, thus the corresponding change of sign. Observe also that as a consequence of Lemma \ref{lgp2} we have that $M_h(x)$ are indeed polynomials.

For the third \eqref{ass2} it is enough to see (see again the proof of Lemma 5.2 of \cite{ddI2}) that
\begin{equation}\label{3eq}
\I^{2a-1}\left(\sigma_{x-\frac{m-1}{2}}R(x)\right)+\sigma_{x+\frac{m+1}{2}}R(x+m)=0.
\end{equation}
First we claim that if $R$ is a polynomial satisfying \eqref{EQR2}, then $R$ is invariant under the action of $\I^{2a-m-1}$. If that is true, then it is easy to check that
\begin{equation*}\label{relR}
\I^{2a+i}\left(R(x-j)\right)=R(x+m+i+j+1).
\end{equation*}
In particular, for $i=-1$ and $j=0$ we have $R(-x-2a)=R(x+m)$. Using this and \eqref{Iprop1} then it is straightforward to see that \eqref{3eq} holds.

For the claim, assume that \eqref{EQR2} holds for every $x$. Applying $\I^{2a-m-1}$ to this identity and using \eqref{Iprop1} we have
\begin{equation}\label{eqdd}
\sigma_{x-\frac{m-1}{2}}\I^{2a-m-1}\left(R(x)\right)+\sigma_{x-1-\frac{m-1}{2}}\I^{2a-m-1}\left(R(x+1)\right)=-\I^{2a-m-1}\left(H(x)\Omega_1^{a,\F}(x)\right).
\end{equation}
Now from \eqref{Om1}, \eqref{condxND} and \eqref{HH2} we have that
$$
H(x)\Omega_1^{a,\F}(x)=\frac{\sigma_{x-m_1-m_3-1}}{p(x)q(x)}\left|
  \begin{array}{@{}c@{}lccc@{}c@{}}
    & &&\hspace{-2.5cm}{}_{j=1,\ldots , m_1+m_3} \\
    \dosfilas{ N_{x-2j+1,m-2j}D_{x-1,2j-2}Y_{u}(\theta_{x-2j+1}) }{u\in U} \\
    \dosfilas{Z_{u}(\theta_{x-2j+1})}{u\in V}
  \end{array}\hspace{-.5cm}\right|,
$$
and therefore, using \eqref{Iprop1}, \eqref{IpropND} and \eqref{invpq}, we have
\begin{align*}
\I^{2a-m-1}&\left(H(x)\Omega_1^{a,\F}(x)\right)\\&=\frac{\sigma_{x-m_1-m_3}}{(-1)^{\binom{m_1+m_3+2}{2}}p(x-1)q(x-1)}\left|
  \begin{array}{@{}c@{}lccc@{}c@{}}
    & &&\hspace{-2.5cm}{}_{j=1,\ldots , m_1+m_3} \\
    \dosfilas{ D_{x-2,m-2j}N_{x+2j-m-2,2j-2}Y_{u}(\theta_{x-m+2j-2}) }{u\in U} \\
    \dosfilas{Z_{u}(\theta_{x-m+2j-2})}{u\in V}
  \end{array}\hspace{-.5cm}\right|\\
  &=\frac{\sigma_{x-m_1-m_3}(-1)^{\binom{m_1+m_3}{2}}}{(-1)^{\binom{m_1+m_3+2}{2}}p(x-1)q(x-1)}\left|
  \begin{array}{@{}c@{}lccc@{}c@{}}
    & &&\hspace{-2.5cm}{}_{j=1,\ldots , m_1+m_3} \\
    \dosfilas{ D_{x-2,2j-2}N_{x-2j,m-2j}Y_{u}(\theta_{x-2j}) }{u\in U} \\
    \dosfilas{Z_{u}(\theta_{x-2j})}{u\in V}
  \end{array}\hspace{-.5cm}\right|\\
&=-H(x-1)\Omega_1^{a,\F}(x-1).
\end{align*}
Note that we interchanged all of the columns in the determinant, thus the corresponding change of sign. Using the same identity \eqref{EQR2} for $x=x-1$ and subtracting it to \eqref{eqdd} we get
\begin{align*}
\sigma_{x-\frac{m-1}{2}}&\left(R(x)-\I^{2a-m-1}\left(R(x)\right)\right)+\sigma_{x-1-\frac{m-1}{2}}\left(R(x-1)-\I^{2a-m-1}\left(R(x-1)\right)\right)=0.
\end{align*}
If we call $Q(x)=\sigma_{x-\frac{m-1}{2}}\left(R(x)-\I^{2a-m-1}\left(R(x)\right)\right)$ then the above equality becomes $Q(x)+Q(x-1)=0$ for every $x$. The solution is given by $Q(x)=Q(0)(-1)^x$, but $Q$ is a polynomial, so the only possibility is to have $Q(0)=0$ and consequently $Q(x)=0$ for every $x$. That means, since $\sigma_x\neq 0$, that $R(x)=\I^{2a-m-1}\left(R(x)\right)$, i.e. $R$ is invariant under the action of $\I^{2a-m-1}$.

Finally, the order of $D_{S}$ in \eqref{Dq2} is given by $2\deg P_S$. This can be proved in a similar way as it was proven in Theorem 6.2 of \cite{ddI2}, now using Lemma \ref{lgp2}. Observe that $H(x)\Omega_1^{a,\F}(x)=\sigma_{x-m_1-m_3+1}P(x)$, where $P$ is the polynomial defined by \eqref{PPP}. Therefore $\deg \left(H(x)\Omega_1^{a,\F}(x)\right)=d+1$, where $d$ is given by \eqref{degP}. From \eqref{EQR2} we have that $\deg R=d$. On the other hand, from \eqref{ass2} we have that $2\deg P_S-1=\deg R+1=d+1$. Therefore $2\deg P_S=d+2$. Using  \eqref{relF} and \eqref{degP} we have after easy computations that
that
\begin{align*}
d+2&=2\left(\sum_{u\in U,V}u-\binom{m_1}{2}-\binom{m_3}{2}+1\right)=2\left(\sum_{f\in F}f-\binom{k_1}{2}-\binom{k_3}{2}+1\right).
\end{align*}
\end{proof}

\section{Examples}
\

\noindent
1. Consider the example with $F$ and $G$ given by
$$
F=\{1,2,\ldots,k\},\quad G=\emptyset.
$$
Then $U=I(F)=\{k\}, V=I(G)=\emptyset$ and $m_1=1, m_3=0$. In \cite{ddI2}, the rational function from which we produce the difference operator is 
\begin{equation}\label{iii}
S(x)=\sigma_{x-1/2}\frac{P(x)}{\Omega ^{a,\F} (x)},
\end{equation}
where $P$ is the polynomial $P(x)=R_k^{-a,-a,2a+N}(\theta_{x-1}+2a)R_k^{-a,-a,2a+N}(\theta_{x-2}+2a)$. This gives a difference operator of order $4k+2$ (according to Theorem \ref{cori1}). Theorem \ref{Th2} provides a better choice for the rational function $S$. Indeed, the function $\Omega_1^{a,\F}(x)$ in \eqref{Om1} is given by
$$
\Omega_1^{a,\F}(x)=R_k^{-a,-a,2a+N}(\theta_{x-1}+2a),
$$
while the function $H(x)$ in \eqref{HH2} is given by $H(x)=\sigma_x$. Therefore we have to find a solution of the difference equation \eqref{EQR2}, i.e.
$$
\sigma_{x-1/2}R(x)+\sigma_{x+1/2}R(x+1)=\sigma_xR_k^{-a,-a,2a+N}(\theta_{x-1}+2a).
$$
Surprisingly enough, the solution in this case is given also by a dual Hahn polynomial. Indeed,
$$
R(x)=\frac{2a+N+1}{2(2a+N-k+1)}R_k^{-a,-a-1,2a+N+1}\left((x-1)(x+2a-1)\right).
$$
The new function $S$ is then defined by $S(x)=\sigma_{x-1/2}R(x)/\Omega^{a,\F} (x)$ (compare with (\ref{iii}) above). With this we have $M_1(x)=M_2(x)=\sigma_{x+1}/2$ and therefore $\tilde M_1(x)=\tilde M_2(x)=1/2$. From Theorem \ref{Th2} we have that the order of the new difference operator is $2k+2$, instead of $4k+2$ as before.

\medskip

\noindent
2. Consider the example with $F$ and $G$ given by
$$
F=\emptyset,\quad G=\{1,2,\ldots,k\}.
$$
Proceeding as before, we have now
$
\Omega_1^{a,\F}(x)=R_k^{-a,-a,-2-N}(\theta_{x-1}+2a),
$
and the solution of \eqref{EQR2} is given again by a dual Hahn polynomial. Indeed,
$$
R(x)=\frac{N+1}{2(N+k+1)}R_k^{-a,-a-1,-1-N}\left((x-1)(x+2a-1)\right).
$$
Again, from Theorem \ref{Th2} we have that the order of the difference operator is $2k+2$, instead of $4k+2$ (see Theorem \ref{cori1}).

\medskip

\noindent
3. Consider the example with $F$ and $G$ given by
$$
F=\{1,2,\ldots,k_1\},\quad G=\{1,2,\ldots,k_2\}.
$$
Then $U=I(F)=\{k_1\}$, $V=I(G)=\{k_2\}$ and $m_1=m_3=1$. The function $\Omega_1^{a,\F}(x)$ in \eqref{Om1} is given by
$$
\Omega_1^{a,\F}(x)=
\begin{vmatrix}
\xi_{x-1,3}^1R_{k_1}^{-a,-a,2a+N}(\theta_{x-1}+2a)&\xi_{x-3,1}^1R_{k_1}^{-a,-a,2a+N}(\theta_{x-3}+2a)\\
R_{k_2}^{-a,-a,-2-N}(\theta_{x-1}+2a)&R_{k_2}^{-a,-a,-2-N}(\theta_{x-3}+2a)
\end{vmatrix},
$$
while the rational function $H(x)$ in \eqref{HH2} is given by
$$
H(x)=\frac{D_{x-1,3}}{N_{x-3,1}}.
$$
Therefore, using  \eqref{condxND}, we have for $R$ the difference equation \eqref{EQR2}
\begin{equation}\label{spmde}
\sigma_{x-3/2}R(x)+\sigma_{x-1/2}R(x+1)=\begin{vmatrix}
N_{x-1,2}R_{k_1}^{-a,-a,2a+N}(\theta_{x-1}+2a)&D_{x-1,2}R_{k_1}^{-a,-a,2a+N}(\theta_{x-3}+2a)\\
R_{k_2}^{-a,-a,-2-N}(\theta_{x-1}+2a)&R_{k_2}^{-a,-a,-2-N}(\theta_{x-3}+2a)
\end{vmatrix},
\end{equation}
where $N_{x,j}$ and $D_{x,j}$ are the polynomials defined by \eqref{xi11}.

With this we have
\begin{align*}
M_1(x)&=M_2(x)=\frac{1}{2}\left(N_{x,2}R_{k_2}^{-a,-a,-2-N}(\theta_{x-2}+2a)-D_{x+2,2}R_{k_2}^{-a,-a,-2-N}(\theta_{x+2}+2a)\right),\\
M_3(x)&=M_4(x)=\frac{1}{2}\left(-D_{x,2}R_{k_1}^{-a,-a,2a+N}(\theta_{x-2}+2a)+N_{x+2,2}R_{k_1}^{-a,-a,2a+N}(\theta_{x+2}+2a)\right).
\end{align*}
From Theorem \ref{Th2} we have that the order of the difference operator is $2k_1+2k_2+2$, instead of $4k_1+4k_2+2$ (see Theorem \ref{cori1}).

Unlike the previous two cases, we have not been able to find the general solution of (\ref{spmde}) for arbitrary $k_1, k_2$ positive integers. Nevertheless, it is possible to compute $R(x)$ for particular values of $k_1$ and $k_2$.

%\section*{Appendix}

\appendix

\section{}

\begin{proof}[Proof of Lemma \ref{lpqR}] Using \eqref{Om1}, \eqref{xi11}, \eqref{condxND} and Lemma \ref{lemaP}, we have
\begin{align*}
H(x)\Omega_1^{a,\F}(x)&=\frac{\sigma_{x-m_1-m_3+1}\left(D_{x-1,m-1}\right)^{m_1}}{p(x)q(x)\left(N_{x-m+1,1}\right)^{m_1}}
\left|
  \begin{array}{@{}c@{}lccc@{}c@{}}
    & &&\hspace{-2.5cm}{}_{j=1,\ldots , m_1+m_3} \\
    \dosfilas{\displaystyle \frac{N_{x-2j+1,m-2j+1}}{D_{x-2j+1,m-2j+1}}Y_{u}(\theta_{x-2j+1}) }{u\in U} \\
    \dosfilas{ Z_{u}(\theta_{x-2j+1})}{u\in V}
  \end{array}\hspace{-.5cm}\right|\\
  &=\frac{\sigma_{x-m_1-m_3+1}}{p(x)q(x)}
\left|
  \begin{array}{@{}c@{}lccc@{}c@{}}
    & &&\hspace{-2.5cm}{}_{j=1,\ldots , m_1+m_3} \\
    \dosfilas{\displaystyle N_{x-2j+1,m-2j}D_{x-1,2j-2}Y_{u}(\theta_{x-2j+1}) }{u\in U} \\
    \dosfilas{ Z_{u}(\theta_{x-2j+1})}{u\in V}
  \end{array}\hspace{-.5cm}\right|=\sigma_{x-m_1-m_3+1}P(x),
\end{align*}
where $P$ is given by \eqref{PPP}. According to Lemma \ref{lemaP}, $P$ is a polynomial. Therefore $H(x)\Omega_1^{a,\F}(x)$ is a polynomial of degree $d+1$ where $d$ is given by \eqref{degP}.
\end{proof}

\begin{proof}[Proof of Lemma \ref{lgp2}]
For a complex number $s$, if $s$ is a pole of $\Psi _j$, write $c_j\equiv c_j(s)$ for its residue. Otherwise $c_j(s)=0$. We assume the following claim (which we prove later).

\textsl{Claim 1.} Let $s$ be a pole of $\Psi _j$ (hence $c_j\not =0$). Then there is a unique $l$, $l\not =j$, $1\le l\le m$, such that $s=-(a+b+1)/2+(j+l)/2$, $s$ is also a pole of $\Psi _l$ and $(-1)^jc_j+(-1)^lc_l=0$. Write now
$$
E=\left\{(j,l):\mbox{$1\le j<l \le m,$\; and\; $\displaystyle \frac{-(a+b+1)+j+l}{2}$\; is a pole of $\Psi_j$ (and so of $\Psi_l$)}\right\}.
$$
Using the claim, the function $M(x)=\sum_{j=1}^m(-1)^j\Psi_j(x+j)$ can be written in the form
$M(x)=P(x)+Q(x)$, where $P$ is a polynomial and
\begin{equation}\label{spm11}
Q(x)=\sum_{(j,l)\in E}(-1)^{j}c_j\left(\frac{1}{x+j+\frac{a+b+1-(j+l)}{2}}-\frac{1}{x+l+\frac{a+b+1-(j+l)}{2}}\right).
\end{equation}
A simple computation gives $\I^{a+b}(Q)=Q$, where $\I$ is the involution defined by \eqref{inv}. Since $\I^{a+b}(M)=-M$, we have
$$
-(P+Q)=-M=\I^{a+b}(M)=\I^{a+b}(P)+\I^{a+b}(Q)=\I^{a+b}(P)+Q.
$$
And so $2Q=-P-\I^{a+b}(P)$. Since $-P-\I^{a+b}(P)$ is a polynomial, (\ref{spm11}) implies that $Q=0$, and hence $M$ is a polynomial.

We next prove the Claim 1. Indeed, let $s$ be a pole of $\Psi _j$ and write
$$
I_j=\{l: \theta_{s-j}=\theta_{s-l},\; l\not =j,\; 1\le l\le m\}.
$$
Let $p$ be a polynomial satisfying that $p(\theta_{s-j})=1$ and $p(\theta_{s-l})=0$ for $l\not \in I_j$. Taking into account the hypotheses (1) of this Lemma  we have that
$
L_p(x)=\sum_{i=1}^m(-1)^{i}p(\theta_{x-i})\Psi_i(x)
$
is a polynomial. Since the poles of $\Psi_j$ are, at most, simple, we can write
\begin{equation*}
L_p(x)=R(x)+\frac{1}{x-s}\sum_{i=1}^m(-1)^ic_ip(\theta_{x-i}),
\end{equation*}
where $R$ is a rational function which is analytic at $s$. Since $L_p(x)$ is a polynomial, we have that
$\sum_{i=1}^m(-1)^ic_ip(\theta_{x-i})$ has to vanish at $s$. Hence
\begin{equation*}\label{pm00}
0=\sum_{i=1}^m(-1)^ic_ip(\theta_{s-i})=(-1)^jc_j+\sum_{l\in I_j}(-1)^lc_l.
\end{equation*}
Since $\theta_x$ is a polynomial of degree two, $I_j$ has to be either the empty set or a singleton. Since $c_j\not =0$,
we deduce that $I_j$ is just a singleton, and so there is a unique $l$, $l\not =j$, $1\le l\le m$, such that $\theta_{s-j}=\theta_{s-l}$, and then $(-1)^jc_j+(-1)^lc_l=0$; hence $c_l\not=0$ and $s$ is also a pole of $\Psi _l$.
The equation $\theta_{s-j}=\theta_{s-l}$ then gives that $s=-(a+b+1)/2+(j+l)/2$. This proves the Claim 1, and hence the first part of the Lemma.

Assume now that the degree of $L_{x^g}$ is at most $2g+\deg (L_1)$. We have already proved that $M$ is a polynomial, hence we prove that its degree is at most $\deg (L_1)$.
To simplify the notation write $d_g=\deg (L_{x^g})$. By hypothesis, we have $d_g\le 2g+d_0$. 

For $j=1,\ldots , m$, denote by $n_j$ and $d_j$ the degrees of the numerator and the denominator of $\Psi_j$, respectively, and define
$
d=\max\{n_j-d_j,\; j=1,\ldots, m\}.
$
We can then write
\begin{equation}\label{1pm0}
\Psi_j(x)=P_j(x)+R_j(x),
\end{equation}
where $P_j$ is a polynomial of degree at most $d$, and $R_j$ is a rational function satisfying that the degree of its numerator is less than the degree of its denominator. We have already proved that
$\sum_{j=1}^m(-1)^jR_j(x+j)=0$ (it coincides with the rational function $Q$ (\ref{spm11})), hence
\begin{equation}\label{1pm1}
M(x)=\sum_{j=1}^m(-1)^j\Psi_j(x+j)=\sum_{j=1}^m(-1)^jP_j(x+j).
\end{equation}
We write
\begin{equation}\label{1pm2}
P_j(x)=\sum_{i=0}^{d}a_i^{j}x^i,
\end{equation}
and finally define
\begin{equation}\label{defc}
\gamma ^l_i=\sum_{j=1}^m (-1)^jj^la_i^j, \quad 0\le l, \quad 0\le i\le d.
\end{equation}
We next assume the following claim.

\textit{Claim 2}. $\gamma_i^{i-l}=0$ for $i=l,\ldots, d$ and $l=d_0+1,\ldots , d$.

The second part of the Lemma follows by inserting the expansion of $P_j$ (\ref{1pm2}) in (\ref{1pm1}) and using the Claim 2:
\begin{align*}
M(x)&=\sum_{j=1}^m(-1)^{j}\sum_{i=0}^{d}a_i^{j}(x+j)^i=\sum_{j=1}^m(-1)^{j}\sum_{i=0}^{d}a_i^{j}\sum_{l=0}^i\binom{i}{l}x^lj^{i-l}\\
&=\sum_{l=0}^{ d}x^l\sum_{i=l}^{d}\binom{i}{l}\sum_{j=1}^m(-1)^{j}j^{i-l}a_i^{j}=\sum_{l=0}^{ d}x^l\sum_{i=l}^{d}\binom{i}{l}\gamma_i^{i-l}=\sum_{l=0}^{d_0}x^l\sum_{i=l}^{d}\binom{i}{l}\gamma_i^{i-l}.
\end{align*}
Claim 2 is equivalent to the following statement:
\begin{equation}\label{1pm4}
\mbox{$\gamma_{d-h+i}^{i}=0$ for $h=i,\ldots, d-d_0-1$ and $i=0,\ldots , d-d_0-1$.}
\end{equation}
We prove (\ref{1pm4}) by induction on $r=d-d_0-1$. Using (\ref{1pm0}), we have
\begin{align*}
L_{x^g}(x)&=\sum_{j=1}^m(-1)^j\theta _{x-j}^g\Psi_j(x)\\
&=\sum_{j=1}^m(-1)^j\theta _{x-j}^gP_j(x)+\sum_{j=1}^m(-1)^j\theta _{x-j}^gQ_j(x)=\sum_{j=1}^m(-1)^j\theta _{x-j}^gP_j(x)+\tilde Q_g(x),
\end{align*}
where $\tilde Q_g=\sum_{j=1}^m(-1)^j\theta _{x-j}^gQ_j(x)$ is certain rational function. Since $L_{x^g}$ is a polynomial, we have that $\tilde Q_g$ has also to be a polynomial, and since  each $Q_j$ is a  rational function satisfying that the degree of its numerator is less than the degree of its denominator, then the degree of $\tilde Q_g$ is at most $2g-1$ (for $g=0$ this means that $Q_0=0$).
Taking into account that $\theta_{x-j}=(x-j)^2+u(x-j)$, where $u=a+b+1$, we can write (after some computations):
\begin{align*}
\sum_{j=1}^m(-1)^j&\theta _{x-j}^gP_j(x)=(-u)^g\sum_{l=0}^{2g}(-1)^lx^l\sum_{v=\max \{0,l-g\}}^g(-u)^{-v}\binom{g}{v}\binom{g+v}{l}\sum_{j=1}^m(-1)^jj^{g+v-l}P_j(x).
\end{align*}
This gives
\begin{equation}\label{1pm3}
L_{x^g}(x)=Q_g(x)+(-u)^g\sum_{l=0}^{2g}(-1)^lx^l\sum_{v=\max \{0,l-g\}}^g(-u)^{-v}\binom{g}{v}\binom{g+v}{l}\sum_{j=1}^m(-1)^jj^{g+v-l}P_j(x).
\end{equation}
Assume now that $r=0$, i.e., $d=d_0+1$. For $g=0$ equation (\ref{1pm3}) gives (recall that $Q_0=0$)
$L_{1}(x)=\sum_{j=1}^m(-1)^jP_j(x)$.
Since $\deg (L_1)=d_0$, we have from (\ref{1pm2}) and (\ref{defc}) that $\gamma_{d_0+1}^0=0$. This is precisely (\ref{1pm4}) for $r=0$.

The induction hypothesis then says $\gamma_{d-h+i}^{i}=0$ for $h=i,\ldots, r-1$ and $i=0,\ldots , r-1$. We have to prove that also
$\gamma_{d-r+l}^{l}=0$ for $l=0,\ldots, r$. To do that, we compute the coefficient $\zeta_{2g-r+d}$ of the power $x^{2g-r+d}$
of $L_{x^g}$ for $g=0,\ldots , r$. Indeed, using (\ref{1pm2}), (\ref{defc}), (\ref{1pm3}), taking into account that each $P_i$ and $Q_g$ have degrees at most $d$ and $2g-1$, respectively,  we have that
$$
\zeta_{2g-r+d}=(-1)^gu^g\sum_{l=2g-r}^{2g}(-1)^l\sum_{v=\max \{0,l-g\}}^g(-u)^{-v}\binom{g}{v}\binom{g+v}{l}\gamma_{2g-r+d-l}^{g+v-l}.
$$
The induction hypothesis says that for $v<g$, $\gamma_{2g-r+d-l}^{g+v-l}=0$, and then
$$
\zeta_{2g-r+d}=\sum_{l=2g-r}^{2g}(-1)^l\binom{2g}{l}\gamma^{2g-l}_{2g-r+d-l}.
$$
Since the degree of $L_{x^g}$ is at most $2g+d_0$ and $2g-r+d=2g-d+d_0+1+d=2g+d_0+1$, we deduce that $\zeta_{2g-r+d}=0$ for $g=0,\ldots ,r$. Writing $x_h=\gamma^{h}_{h-r+d}$, this gives the following homogeneous system of $r+1$ equations
$$
\sum_{h=0}^{r}(-1)^h\binom{2g}{h}x_h=0.
$$
Since $\det ((-1)^h\binom{2g}{h})_{g,h=0,\ldots ,r}=(-2)^{\binom{r}{2}}\not =0$, we deduce that $\gamma^{h}_{h-r+d}=0$, $h=0,\ldots ,r$. This completes the proof of the Claim 2.
\end{proof}


\begin{thebibliography}{9999}

\bibitem{Kr2}\textrm{H. L. Krall},
On orthogonal polynomials satisfying a certain fourth order differential equation,
The Pennsylvania State College Studies, No. 6, 1940.

\bibitem{DG}
J. J. Duistermaat and F. A. Gr\"unbaum,
Differential equations in the spectral parameter,
Comm. Math. Phys. \textbf{103} (1986), 177--240.

\bibitem{GrH1}  F. A. Gr\"unbaum and L. Haine,
Orthogonal polynomials satisfying differential equations: the role of the Darboux transformation,
in: D. Levi, L. Vinet, P. Winternitz (Eds.), Symmetries an Integrability of Differential Equations, CRM Proc. Lecture Notes, vol. 9, Amer. Math. Soc. Providence, RI, 1996, 143--154.

\bibitem{GrH3}
F. A. Gr\"unbaum and L. Haine,
Bispectral Darboux transformations: an extension of the Krall polynomials, Int. Math. Res. Not. \textbf{8} (1997), 359--392.

\bibitem{koekoe} J. Koekoek and R. Koekoek,
On a differential equation for Koornwinder's generalized Laguerre polynomials,
Proc. Amer. Math. Soc. \textbf{112} (1991), 1045--1054.

\bibitem{koe} R. Koekoek,
Differential Equations for Symmetric Generalized Ultraspherical Polynomials,
Trans. Amer. Math. Soc. \textbf{345} (1994), 47--72.

\bibitem{koekoe2} J. Koekoek and R. Koekoek,
Differential equations for generalized Jacobi polynomials,
J. Compt. Appl. Math. \textbf{126} (2000), 1--31.


\bibitem{L1} L. L. Littlejohn,
The Krall polynomials: a new class of orthogonal polynomials,
Quaest. Math.  {\bf 5} (1982), 255--265.


\bibitem{L2} L. L. Littlejohn,
An application of a new theorem on orthogonal polynomials and differential equations,
Quaest. Math.  {\bf 10} (1986), 49--61.


\bibitem{GrHH}
F. A. Gr\"unbaum, L. Haine and E. Horozov,
Some functions that generalize the Krall-Laguerre polynomials,
J. Comput. Appl. Math. \textbf{106}, 271--297 (1999).

\bibitem{GrY}
F. A. Gr\"unbaum and M. Yakimov,
Discrete bispectral Darboux transformations from Jacobi operators.
Pac. J. Math., \textbf{204}, 395--431 (2002).


\bibitem{Plamen1} P. Iliev,
Krall-Jacobi commutative algebras of partial differential operators,
J. Math. Pures Appl. \textbf{96} (2011), 446--461.

\bibitem{Plamen2} P. Iliev,
Krall-Laguerre commutative algebras of ordinary
differential operators,
Ann. Mat. Pur. Appl. \textbf{192} (2013), 203--224.

\bibitem{Zh} A. Zhedanov,
A method of constructing Krall's polynomials,
J. Compt. Appl. Math. \textbf{107} (1999), 1--20.


\bibitem{BGR} C. Brezinski, L. Gori and A. Ronveaux (Eds.),
Orthogonal polynomials and their applications,
IMACS Annals on Computing and Applied Mathematics (No. 9), J.C. Baltzer AG, Basel, 1991.


\bibitem{du0} A. J. Dur\'an,
Orthogonal polynomials satisfying  higher-order difference equations, Constr. Approx. \textbf{36} (2012), 459--486.

\bibitem{du1} A. J. Dur\'an,
Using $\D$-operators to construct orthogonal polynomials satisfying  higher-order difference or differential equations, J. Approx. Theory \textbf{174} (2013), 10--53.

\bibitem{ddI} A. J. Dur\'an and M. D. de la Iglesia,
Constructing bispectral orthogonal polynomials  from the classical
discrete families of Charlier, Meixner and Krawtchouk, Constr. Approx. \textbf{41} (2015), 49--91.

\bibitem{ddI2} A. J. Dur\'an and M. D. de la Iglesia,
Constructing Krall-Hahn orthogonal polynomials, J. Math. Anal. Appl. \textbf{424} (2015), 361--384.


\bibitem{AD}  R. \'Alvarez-Nodarse and A. J. Dur\'an,
Using $\D$-operators to construct orthogonal polynomials satisfying higher-order $q$-difference equations, J. Math. Anal. Appl. \textbf{424} (2015) 304--320.


\bibitem{ddI1} A. J. Dur\'an and M. D. de la Iglesia,
Differential equations for discrete Laguerre-Sobolev orthogonal polynomials, J. Approx. Theory \textbf{195} (2015), 70--88.


\bibitem{dudh} A. J. Dur\'an,
Constructing bispectral dual Hahn polynomials,  J. Approx. Theory \textbf{189} (2015) 1--28.


\bibitem{ddI3} A. J. Dur\'an and M. D. de la Iglesia,
Differential equations for discrete Jacobi-Sobolev orthogonal polynomials, 
J. Spectr. Theory \textbf{8} (2018), 191--234.



\end{thebibliography}
\end{document}